\documentclass[a4paper]{amsart}
\usepackage[alphabetic]{amsrefs}
\usepackage{amssymb,mathrsfs}
\usepackage[all]{xy}
\newdir{ >}{{}*!/-5pt/@{>}} 
\SelectTips{cm}{} 
\usepackage{amssymb}
\usepackage{tikz}

\title{Exponentials of non-singular simplicial sets}
\author{Vegard Fjellbo}
\address{Department of Mathematics, University of Oslo, Norway}
\email{vegard.fjellbo@gmail.com}
\author{John Rognes}
\address{Department of Mathematics, University of Oslo, Norway}
\email{rognes@math.uio.no}

\date{May 16th 2021}

\newtheorem{theorem}{Theorem}[section]
\newtheorem{proposition}[theorem]{Proposition}
\newtheorem{lemma}[theorem]{Lemma}
\theoremstyle{definition}

\theoremstyle{remark}
\newtheorem{remark}[theorem]{Remark}

\numberwithin{equation}{section}

\newcommand{\Ex}{Ex}
\newcommand{\longto}{\longrightarrow}
\newcommand{\nsSet}{nsSet}

\newcommand{\Sd}{Sd}
\newcommand{\sSet}{sSet}
\renewcommand{\:}{\colon}
\renewcommand{\epsilon}{\varepsilon}

\begin{document}

\begin{abstract}
A simplicial set is non-singular if the representing map of each
non-degenerate simplex is degreewise injective.  The simplicial mapping
set $X^K$ has $n$-simplices given by the simplicial maps $\Delta[n]
\times K \to X$.  We prove that $X^K$ is non-singular whenever $X$
is non-singular.  It follows that non-singular simplicial sets form a
cartesian closed category with all limits and colimits, but it is not
a topos.
\end{abstract}

\maketitle

\section{Introduction}

Let $\sSet$ be the category of simplicial sets, and let $\nsSet$
denote its full subcategory of \emph{non-singular simplicial sets},
i.e., those $X$ such that for each non-degenerate simplex $x \in X_n$
the representing map $\bar x \: \Delta[n] \to X$ is degreewise injective.
The geometric realization $|X|$ of each non-singular simplicial set
admits a well-defined PL (piecewise-linear) structure, and the category
$\nsSet$ plays a key role in the passage between simplicial sets and
PL manifolds in the proof of the stable parametrized $h$-cobordism
theorem~\cite{WJR13}*{Thm.~0.1, \S3.4}.

The inclusion $U \: \nsSet \to \sSet$ admits a left adjoint $D \: \sSet
\to \nsSet$, called \emph{desingularization}, and the adjunction unit
$\eta_X \: X \to UDX$ is degreewise surjective.  The category $\nsSet$
has all (small) limits and colimits, which are preserved by~$U$ and~$D$,
respectively.  Let $(\Sd, \Ex)$ denote Kan's adjoint pair~\cite{Kan57} of
endofunctors of $\sSet$.  The first author has exhibited a model structure
on the category $\nsSet$, and furthermore shown that the adjunction
$$
D \Sd^2 \: \sSet \rightleftarrows \nsSet \: \Ex^2 U
$$
defines a Quillen equivalence from the standard model structure on
simplicial sets~\cite{Fje}.  The proofs of these two results depend on
knowing that the endofunctor $X \mapsto X \times \Delta[1]$ of $\nsSet$
preserves all colimits, and one purpose of the present paper is to
establish this fact.

For any simplicial sets $X$ and~$K$ let $X^K$ be the
\emph{simplicial mapping set}, with $n$-simplices the set of maps
$\Delta[n] \times K\to X$.  Our main result follows.

\begin{theorem} \label{thm:XKnonsing}
Let $X$ and $K$ be any two simplicial sets.  If $X$ is non-singular,
then so is $X^K$.
\end{theorem}

It follows that $X \mapsto X^K$ restricts to an endofunctor of $\nsSet$.
This implies the following generalization of the aforementioned fact.

\begin{proposition} \label{prop:XKprescolim}
Let $K$ be any non-singular simplicial set.  The endofunctor $X \mapsto
X \times K$ of $\nsSet$ preserves all colimits.
\end{proposition}

The proof of Theorem~\ref{thm:XKnonsing} follows easily from the
following special case, which also directly implies the case~$K =
\Delta[1]$ of Proposition~\ref{prop:XKprescolim}.

\begin{proposition} \label{prop:XD1nonsing}
If $X$ is non-singular, then so is $X^{\Delta[1]}$.
\end{proposition}

We may restate Theorem~\ref{thm:XKnonsing} by saying that the
non-singular simplicial sets form an \emph{exponential ideal} in the
cartesian closed category~\cite{ML98}*{\S IV.6} of simplicial sets.
The adjunction $(D, U)$ exhibits $\nsSet$ as a \emph{reflective} full
subcategory~\cite{ML98}*{\S IV.3} of $\sSet$, which is \emph{closed
under exponentiation} in the sense of~\cite{Day72}.  In this situation,
Day's reflection theorem~\cite{Day72}*{Thm.~1.2, Cor.~2.1} shows that
the reflector $D \: \sSet \to \nsSet$ preserves finite products, making
$\nsSet$ a cartesian closed category.

\begin{proposition} \label{prop:Dfinprod}
Desingularization $D \: \sSet \to \nsSet$ preserves finite products.
\end{proposition}

\begin{remark}
The category $\nsSet$ is not a \emph{topos} in the sense
of~\cite{ML98}*{\S IV.10}, because it does not admit a subobject
classifier $t \: \Delta[0] \to \Omega$.  Here $\Omega_0$ would
have to consist of precisely two elements, so $\Omega$ would be at
most $1$-dimensional, and could not classify all the subobjects of
$\Delta[2]$.  This is related to the fact that desingularization does not
in general preserve equalizers, as the example of the two maps $\Delta[0]
\rightrightarrows \Delta[2]/\delta_1\Delta[1]$ illustrates.
\end{remark}

We give the proof of Proposition~\ref{prop:XD1nonsing} in
Section~\ref{sec:rigid}, and deduce the remaining results in
Section~\ref{sec:outstanding}.

\section{Rigidity of prisms} \label{sec:rigid}

Informally, Proposition~\ref{prop:XD1nonsing} asserts that maps $\Phi \:
\Delta[n] \times \Delta[1] \to X$ from prisms to non-singular simplicial
sets are very rigid.

We recall some terminology and notation before turning to the proofs.
For each $n\ge0$ let $[n]$ denote the totally ordered set $\{0 < 1 < \dots
< n\}$. Following \cite{FP90}*{\S 4.1}, we shall refer to the functions
$\alpha \: [m] \to [n]$ such that $\alpha(i) \le \alpha(j)$ for all $i
\le j$ as \emph{operators}.  These are the objects and morphisms of the
category~$\Delta$.  For a simplicial set~$X \: [n] \mapsto X_n$ we write
$x\alpha \in X_m$ for the value of the operator~$\alpha \: [m] \to [n]$
on a simplex~$x \in X_n$.  The \emph{standard $n$-simplex} $\Delta[n]$
is the simplicial set $[m] \mapsto \Delta([m], [n])$ represented by~$[n]$.

An injective operator is said to be a \emph{face operator}, and a
surjective operator is said to be a \emph{degeneracy operator}. Special
face operators are the \emph{elementary face operators} $\delta^n_i \:
[n-1] \to [n]$ that omit the element~$i$, and the \emph{vertex operators}
$\epsilon^n_i \: [0] \to [n]$ that hit the element~$i$. Special degeneracy
operators are the \emph{elementary degeneracy operators} $\sigma^n_i \:
[n+1] \to [n]$ that send $i$ and its successor~$i+1$ to~$i$.  Usually,
we omit the superscript in the notation.

A face operator or degeneracy operator is \emph{proper} if it is not
the identity. A simplex~$x$ is a \emph{(proper) face} of a simplex~$y$
if $x = y\mu$ for a (proper) face operator~$\mu$. Analogously, $x$
is a \emph{(proper) degeneracy} of $y$ if $x = y\rho$ for a (proper)
degeneracy operator~$\rho$.  A simplex is \emph{degenerate} if it
is a proper degeneracy of some simplex. Otherwise, it is said to be
\emph{non-degenerate}.

By the Eilenberg--Zilber lemma \cite{FP90}*{Thm.~4.2.3} any simplex $x$
in a simplicial set~$X$ can be uniquely expressed as a degeneration $x =
x^\sharp x^\flat$ of a non-degenerate simplex.  We call the non-degenerate
simplex $x^\sharp$ the \emph{non-degenerate part} of~$x$, and will refer
to the degeneracy operator $x^\flat$ as the \emph{degenerate part} of $x$.
By the Yoneda lemma, the $n$-simplices $x$ of a simplicial set $X$ are
in natural bijective correspondence with the simplicial maps $\bar x \:
\Delta[n] \to X$. The map $\bar x$ is the \emph{representing map} of $x$.

\begin{lemma}
Let $x \in X_n$ be any simplex.  The representing map $\bar x \: \Delta[n]
\to X$ is degreewise injective if and only if the $n+1$ vertices $x
\epsilon_0, \dots, x \epsilon_n \in X_0$ are pairwise distinct.  \qed
\end{lemma}

\begin{lemma} \label{lem:xflatthroughrho}
Let $x$ be a simplex in a non-singular simplicial set~$X$, and suppose
that $x \epsilon_k = x \epsilon_l$ for some $k < l$.  Then the
degenerate part $x^\flat$ of $x$ factors uniquely through the proper
degeneracy operator $\sigma_k \cdots \sigma_{l-1}$.
\end{lemma}

\begin{proof}
The representing map of the non-degenerate part $x^\sharp$ is degreewise
injective, since $X$ is non-singular, so its vertices are pairwise
distinct.  It follows that $x^\flat(k) = x^\flat(l)$.  Since $x^\flat$
is order-preserving, it also follows that $x^\flat(k) = x^\flat(j)$
for all $k \le j \le l$.  Let $\rho = \sigma_k \cdots \sigma_{l-1}$.
Then $x^\flat(i) = x^\flat(j)$ whenever $\rho(i) = \rho(j)$, and
this implies that $x^\flat  = (x^\flat \mu) \rho$, where $\mu$ is any
choice of section to~$\rho$.  Thus the asserted factorization exists.
Its uniqueness is automatic, since $\rho$ is surjective.
\end{proof}

\begin{proof}[Proof of Proposition~\ref{prop:XD1nonsing}]
Suppose that $X$ is non-singular.  We must show that each non-degenerate
$n$-simplex $\Phi$ in the simplicial mapping set $X^{\Delta[1]}$ has
$n+1$ distinct vertices $\Phi \epsilon_0, \dots, \Phi \epsilon_n$.
Equivalently, we must show that if the $k$-th and $l$-th vertices of an
$n$-simplex $\Phi$ are equal, for some $0 \le k < l \le n$, then $\Phi$
is degenerate.  This follows from the two lemmas below.
\end{proof}

\begin{lemma} \label{lem:Phiepskepsl}
Suppose that $X$ is non-singular and $\Phi$ is an $n$-simplex in
$X^{\Delta[1]}$ such that $\Phi \epsilon_k = \Phi \epsilon_l$, for some
$0 \le k < l \le n$.  Then
$$
\Phi \epsilon_k = \Phi \epsilon_j = \Phi \epsilon_l
$$
for all $k \le j \le l$.
\end{lemma}

\begin{lemma} \label{lem:PhiPsisigmak}
Suppose that $X$ is non-singular and $\Phi$ is an $n$-simplex in
$X^{\Delta[1]}$ such that $\Phi \epsilon_k = \Phi \epsilon_{k+1}$, for some
$0 \le k < n$.  Then there is an $(n-1)$-simplex $\Psi$ in $X^{\Delta[1]}$
for which $\Phi = \Psi \sigma_k$, exhibiting $\Phi$ as a degenerate
simplex.
\end{lemma}

We introduce some more notation before proving these lemmas.
By definition, an $n$-simplex in $X^{\Delta[1]}$ is a simplicial
map
$$
\Phi \: \Delta[n] \times \Delta[1] \longto X \,.
$$
Here, the \emph{prism} $\Delta[n] \times \Delta[1]$ is generated by the
non-degenerate $(n+1)$-simplices
$$
\gamma^{n+1}_j \: \Delta[n+1] \longto \Delta[n] \times \Delta[1] \,,
$$
for $0 \le j \le n$, given by
$$
\gamma^{n+1}_j(i) = \begin{cases}
(i,0) & \text{for $0 \le i \le j$,} \\
(i-1,1) & \text{for $j+1 \le i \le n+1$.}
\end{cases}
$$
Viewing $\Delta[n] \times \Delta[1]$ as the nerve of the partially
ordered set $[n] \times [1]$, these generators can be seen as maximal
length paths in the diagram below.
$$
\xymatrix{
(0,1) \ar[r] & \dots \ar[r]
	& (j,1) \ar[r] & (j+1, 1) \ar[r]
	& \dots \ar[r] & (n,1) \\
(0,0) \ar[r] \ar[u] & \dots \ar[r]
	& (j,0) \ar[r] \ar[u] \ar[ur] & (j+1, 0) \ar[r] \ar[u]
	& \dots \ar[r] & (n,0) \ar[u]
}
$$
In particular, they satisfy the relations
\begin{equation} \label{eq:gammadeltarel}
\gamma^{n+1}_j \delta_{j+1} = \gamma^{n+1}_{j+1} \delta_{j+1}
\end{equation}
for $0 \le j < n$.  Conversely, to specify $\Phi$ it suffices to give its
values $\Phi \gamma^{n+1}_j$ on these $n+1$ generators, subject to the $n$
relations $(\Phi \gamma^{n+1}_j) \delta_{j+1} = (\Phi \gamma^{n+1}_{j+1})
\delta_{j+1}$.

\begin{proof}[Proof of Lemma~\ref{lem:Phiepskepsl}]
Let $X$ be non-singular and let $\Phi$ be an $n$-simplex in
$X^{\Delta[1]}$ with $\Phi \epsilon_k = \Phi \epsilon_l$,
where $0 \le k < l \le n$.
The vertex operators $\epsilon_i \: \Delta[0] \to \Delta[1]$ for $i \in
\{0,1\}$ induce maps $\epsilon_i^* \: X^{\Delta[1]} \to X^{\Delta[0]}
\cong X$.  Let $x_i = \epsilon_i^* \Phi$ in $X_n$ be represented by
the composite
$$
\bar x_i \: \Delta[n] \cong \Delta[n] \times \Delta[0]
	\overset{1\times\epsilon_i}\longto \Delta[n] \times \Delta[1]
	\overset{\Phi}\longto X \,,
$$
restricting $\Phi$ to the bottom (for $i=0$) or the top (for $i=1$)
of the prism.  The hypothesis on~$\Phi$ implies that $x_i \epsilon_k
= x_i \epsilon_l$ in $X_0$, so by Lemma~\ref{lem:xflatthroughrho} we
can factor the degenerate part $x_i^\flat$ of $x_i$ through $\sigma_k
\cdots \sigma_{l-1}$, so that $x_i = y_i \sigma_k \cdots \sigma_{l-1}$
for some $(n+k-l)$-simplices $y_i$ of~$X$.

Consider any~$j$ with $k \le j < l$, let $\mu \: [1] \to [n]$ be the
face operator given by $\mu(0) = j$ and $\mu(1) = j+1$, and view the
$1$-simplex $\Phi \mu$ in $X^{\Delta[1]}$ as the map $\Delta[1] \times
\Delta[1] \to X$ indicated by the following square.
$$
\xymatrix@+2pc@!C{
x_1 \epsilon_j \ar[r]^-{x_1\mu}
		\ar@{}[dr]|(0.3){z_0}
	& x_1 \epsilon_{j+1} \\
x_0 \epsilon_j \ar[r]_-{x_0\mu} \ar[u]^-{\Phi \epsilon_j} \ar[ur]
	& x_0 \epsilon_{j+1} \ar[u]_-{\Phi \epsilon_{j+1}}
		\ar@{}[ul]|(0.3){z_1}
}
$$
The factorization of $x_i$ through $\sigma_j$ shows that $x_i \epsilon_j =
x_i \epsilon_{j+1}$, for each~$i$.  Hence each $2$-simplex $z_i$ does not
have pairwise distinct vertices, and must therefore be degenerate, since
$X$ is non-singular.  By Lemma~\ref{lem:xflatthroughrho} we must have $z_0
= w_0 \sigma_1$ and $z_1 = w_1 \sigma_0$ for some $1$-simplices $w_i$.
More precisely, we must have $w_0 = z_0 \delta_2 = \Phi \epsilon_j$
and $w_1 = z_1 \delta_0 = \Phi \epsilon_{j+1}$.

It follows that the diagonal $1$-simplex in the figure is simultaneously
equal to $z_0 \delta_1 = (\Phi \epsilon_j) \sigma_1 \delta_1 = \Phi
\epsilon_j$ and to $z_1 \delta_1 = (\Phi \epsilon_{j+1}) \sigma_0
\delta_1 = \Phi \epsilon_{j+1}$.  This proves that $\Phi \epsilon_j =
\Phi \epsilon_{j+1}$ are equal as vertices in $X^{\Delta[1]}$.
\end{proof}

\begin{proof}[Proof of Lemma~\ref{lem:PhiPsisigmak}]
Let $X$ be non-singular and let $\Phi$ be an $n$-simplex in
$X^{\Delta[1]}$ with $\Phi \epsilon_k = \Phi \epsilon_{k+1}$, where $0 \le
k < n$.  We will construct an $(n-1)$-simplex $\Psi$ in $X^{\Delta[1]}$
with $\Phi = \Psi \sigma_k$.  Equivalently, we must define $\Psi \:
\Delta[n-1] \times \Delta[1] \to X$ so as to make the right hand triangle
commute in the diagram below.
$$
\xymatrix{
\Delta[n+1] \ar[r]^-{\gamma^{n+1}_j}
	& \Delta[n] \times \Delta[1]
		\ar[r]^-{\Phi} \ar[d]_-{\sigma_k \times 1}
	& X \\
& \Delta[n-1] \times \Delta[1] \ar[ur]_-{\Psi}
}
$$
The triangle will commute if $\Phi \gamma^{n+1}_j = \Psi (\sigma_k
\times 1) \gamma^{n+1}_j$ for each $0 \le j \le n$, since the simplices
$\gamma^{n+1}_0, \dots, \gamma^{n+1}_n$ generate the prism $\Delta[n]
\times \Delta[1]$.  Here
\begin{equation} \label{eq:sigmagamma}
(\sigma_k \times 1) \gamma^{n+1}_j = \begin{cases}
\gamma^n_j \sigma_{k+1} & \text{for $0 \le j \le k$,} \\
\gamma^n_{j-1} \sigma_k & \text{for $k < j \le n$.}
\end{cases}
\end{equation}
Should $\Psi$ exist, it must therefore satisfy
$$
\Phi(\gamma^{n+1}_j) = \begin{cases}
\Psi(\gamma^n_j) \sigma_{k+1} & \text{for $0 \le j \le k$,} \\
\Psi(\gamma^n_{j-1}) \sigma_k & \text{for $k < j \le n$.}
\end{cases}
$$
Observing that $\delta_{k+1}$ is a section to both $\sigma_k$ and
$\sigma_{k+1}$, we are led to define a function
$$
\psi \: \{\gamma^n_0, \dots, \gamma^n_{n-1}\}
	\longto X_n
$$
by
$$
\psi(\gamma^n_j) = \begin{cases}
\Phi(\gamma^{n+1}_j) \delta_{k+1} & \text{for $0 \le j \le k$,} \\
\Phi(\gamma^{n+1}_{j+1}) \delta_{k+1} & \text{for $k \le j \le n-1$,}
\end{cases}
$$
which specifies where $\Psi$ must send the generators $\gamma^n_0, \dots,
\gamma^n_{n-1}$ of~$\Delta[n-1] \times \Delta[1]$, should it exist.
Note that for $j=k$ the relation
$$
\Phi(\gamma^{n+1}_k) \delta_{k+1}
= \Phi(\gamma^{n+1}_k \delta_{k+1})
= \Phi(\gamma^{n+1}_{k+1} \delta_{k+1})
= \Phi(\gamma^{n+1}_{k+1}) \delta_{k+1}
$$
holds, by~\eqref{eq:gammadeltarel}, so $\psi(\gamma^n_k)$ is unambiguously
defined.  To verify that $\Psi(\gamma^n_j) = \psi(\gamma^n_j)$ for $0 \le
j \le n-1$ defines a map $\Psi \: \Delta[n-1] \times \Delta[1] \to X$,
it is (necessary and) sufficient to confirm the relations
\begin{equation} \label{eq:psigammadelta}
\psi(\gamma^n_j) \delta_{j+1} = \psi(\gamma^n_{j+1}) \delta_{j+1}
\end{equation}
for $0 \le j < n-1$.
We separate the proof of~\eqref{eq:psigammadelta} into
two cases.

First, for $0 \le j < k$ we use the general rule $\delta_{k+1}
\delta_{j+1} = \delta_{j+1} \delta_k$ for $j < k$, together
with~\eqref{eq:gammadeltarel}, to see that
$$
\psi(\gamma^n_j) \delta_{j+1}
	= \Phi \gamma^{n+1}_j \delta_{k+1} \delta_{j+1}
	= \Phi \gamma^{n+1}_j \delta_{j+1} \delta_k
$$
is equal to
$$
\psi(\gamma^n_{j+1}) \delta_{j+1}
	= \Phi \gamma^{n+1}_{j+1} \delta_{k+1} \delta_{j+1}
	= \Phi \gamma^{n+1}_{j+1} \delta_{j+1} \delta_k \,.
$$
Second, for $k \le j < n-1$ we use the general rule $\delta_{k+1}
\delta_{j+1} = \delta_{j+2} \delta_{k+1}$ for $k \le j$, together
with~\eqref{eq:gammadeltarel}, to see that
$$
\psi(\gamma^n_j) \delta_{j+1}
	= \Phi \gamma^{n+1}_{j+1} \delta_{k+1} \delta_{j+1}
	= \Phi \gamma^{n+1}_{j+1} \delta_{j+2} \delta_{k+1}
$$
is equal to
$$
\psi(\gamma^n_{j+1}) \delta_{j+1}
	= \Phi \gamma^{n+1}_{j+2} \delta_{k+1} \delta_{j+1}
	= \Phi \gamma^{n+1}_{j+2} \delta_{j+2} \delta_{k+1} \,.
$$
This concludes the verification of~\eqref{eq:psigammadelta},
giving us a well-defined map $\Psi$.

It still remains to argue that $\Phi = \Psi (\sigma_k \times 1)$, and
this is where we use the hypotheses on~$X$ and $\Phi$.  It suffices to
check that the equation
\begin{equation} \label{eq:PhigammaPsisigmagamma}
\Phi \gamma^{n+1}_j = \Psi (\sigma_k \times 1) \gamma^{n+1}_j
\end{equation}
holds for $0 \le j \le n$.  Again, we separate the proof into two cases.

First, for $0 \le j \le k$ we must show that the $(n+1)$-simplex
$z_j = \Phi(\gamma^{n+1}_j)$ in~$X$ is equal to
$$
\Psi (\sigma_k \times 1) \gamma^{n+1}_j
	= \Psi \gamma^n_j \sigma_{k+1}
	= \Phi(\gamma^{n+1}_j) \delta_{k+1} \sigma_{k+1}
	= z_j \delta_{k+1} \sigma_{k+1} \,,
$$
where we have used the calculation~\eqref{eq:sigmagamma}.
The vertices $z_j \epsilon_{k+1}$ and $z_j \epsilon_{k+2}$ in~$X$
are equal to $\epsilon_1^*(\Phi \epsilon_k)$ and $\epsilon_1^*(\Phi
\epsilon_{k+1})$, respectively, hence are equal by the assumption
that $\Phi \epsilon_k = \Phi \epsilon_{k+1}$.  It follows by
Lemma~\ref{lem:xflatthroughrho} that $z_j = w_j \sigma_{k+1}$
for some $n$-simplex $w_j$ in~$X$, since $X$ is non-singular.
This immediately implies that $z_j \delta_{k+1} \sigma_{k+1} = w_j
\sigma_{k+1} \delta_{k+1} \sigma_{k+1} = w_j \sigma_{k+1} = z_j$, since
$\delta_{k+1}$ is a section to $\sigma_{k+1}$.

Second, for $k < j \le n$ we must show that the $(n+1)$-simplex
$z_j = \Phi(\gamma^{n+1}_j)$ in~$X$ is equal to
$$
\Psi (\sigma_k \times 1) \gamma^{n+1}_j
	= \Psi \gamma^n_{j-1} \sigma_k
	= \Phi(\gamma^{n+1}_j) \delta_{k+1} \sigma_k
	= z_j \delta_{k+1} \sigma_k \,.
$$
The vertices $z_j \epsilon_k$ and $z_j \epsilon_{k+1}$ in~$X$
are equal to $\epsilon_0^*(\Phi \epsilon_k)$ and $\epsilon_0^*(\Phi
\epsilon_{k+1})$, respectively, hence are themselves equal.  It follows
by Lemma~\ref{lem:xflatthroughrho} that $z_j = w_j \sigma_k$ for
some $n$-simplex $w_j$ in~$X$.  This implies that $z_j \delta_{k+1}
\sigma_k = w_j \sigma_k \delta_{k+1} \sigma_k = w_j \sigma_k = z_j$,
since $\delta_{k+1}$ is a section to $\sigma_k$.
This concludes our verification of~\eqref{eq:PhigammaPsisigmagamma},
proving that $\Phi$ is a degenerate simplex of $X^{\Delta[1]}$.
\end{proof}

\section{Outstanding proofs} \label{sec:outstanding}

\begin{proof}[Proof of Theorem~\ref{thm:XKnonsing}]
Let $X$ be any non-singular simplicial set. By
Proposition~\ref{prop:XD1nonsing} and induction, $X^{\Delta[1]^n}$
is non-singular, for each~$n\ge0$.  The inclusion $i \: \Delta[n] \to
\Delta[1]^n$ sending $j \in [n]$ to $(1, \dots, 1, 0, \dots, 0) \in
[1]^n$ (with $j$ copies of~$1$) admits a retraction $r \: \Delta[1]^n
\to \Delta[n]$ sending $(k_1, \dots, k_n)$ to the largest index~$j$
such that $k_j = 1$.  Hence $r^* \: X^{\Delta[n]} \to X^{\Delta[1]^n}$
is split injective, and shows that $X^{\Delta[n]}$ is non-singular.

For any simplicial set~$K$, we can find a simplicial set
$L = \coprod_\alpha \Delta[n_\alpha]$ and a degreewise
surjection $s \: L \to K$.  The induced map
$$
s^* \: X^K \longto X^L \cong \prod_\alpha X^{\Delta[n_\alpha]}
$$
is then degreewise injective, and exhibits $X^K$ as a simplicial subset
of a product of non-singular simplicial sets.  It follows that $X^K$
is non-singular.
\end{proof}

\begin{proof}[Proof of Proposition~\ref{prop:XKprescolim}]
When $X$, $K$ and $Y$ are non-singular, so that $X \times K$ and $Y^K$
are non-singular by Theorem~\ref{thm:XKnonsing}, the natural bijection
$\sSet(X \times K, Y) \cong \sSet(X, Y^K)$ restricts to a natural
bijection $\nsSet(X \times K, Y) \cong \nsSet(X, Y^K)$.  Hence the
endofunctor $X \mapsto X \times K$ of $\nsSet$ is a left adjoint, and
preserves all colimits.
\end{proof}

\begin{proof}[Proof of Proposition~\ref{prop:Dfinprod}]
Let $X$ and~$Y$ be any simplicial sets.  Recall that each adjunction unit
$\eta_Z \: Z \to DZ$ is degreewise surjective.  Let $a \: D(X \times Y)
\to DX \times DY$ be induced by the two projections
from~$X \times Y$.  The composite
$$
X \times Y \overset{\eta_{X \times Y}}\longto D(X \times Y)
	\overset{a}\longto
	DX \times DY
$$
is then equal to $\eta_X \times \eta_Y$, so $a$ is degreewise surjective.
The right adjoint $X \to D(X \times Y)^Y$ of $\eta_{X \times Y}$ factors
through $\eta_X \: X \to DX$, since $D(X \times Y)^Y$ is non-singular
by Theorem~\ref{thm:XKnonsing}.  Hence there is a unique factorization
$$
X \times Y \overset{\eta_X \times 1}\longto DX \times Y
	\overset{b}\longto D(X \times Y)
$$
of $\eta_{X \times Y}$.  Similarly, there is a unique factorization
$$
DX \times Y \overset{1 \times \eta_Y}\longto DX \times DY
	\overset{c}\longto D(X \times Y)
$$
of~$b$, again by Theorem~\ref{thm:XKnonsing}.  It follows that $c a
\eta_{X \times Y} = c (\eta_X \times \eta_Y) = \eta_{X \times Y}$,
so that $c a = 1$, which proves that $a$ is an isomorphism.
\end{proof}

\end{document}